\documentclass[12pt]{article}
\usepackage[utf8]{inputenc}

\usepackage{amsfonts}
\usepackage{amssymb}
\usepackage{amsmath}
\usepackage{amsthm}
\usepackage{amscd}
\usepackage{mathrsfs}
\usepackage[T2A]{fontenc}

\usepackage[english]{babel}

\usepackage{amsthm}

\def \le {\leqslant}
\def \ge {\geqslant}

\theoremstyle{plain}

\newtheorem{lem}{Lemma}[section]

\newtheorem{theor}{Theorem}

\addto\captionsrussian{}
\addto\captionsrussian{}
\newcommand\blfootnote[1]{%
  \begingroup
  \renewcommand\thefootnote{}\footnote{#1}%
  \addtocounter{footnote}{-1}%
  \endgroup
}

\begin{document}

\begin{center}
\textsc{\Large On the derivative of the Minkowski question-mark function}\\ 
\,
\,
\large Dmitry Gayfulin\footnote[1]{Research is supported by RNF grant No. 19-11-00001.}
\end{center}

\begin{abstract}
\noindent
The Minkowski question-mark function $?(x)$ is a continuous strictly increasing function defined on $[0,1]$ interval. It is well known fact that the derivative of this function, if exists, can take only two values: $0$ and $+\infty$. It is also known that the value of the derivative $?'(x)$ at the point $x=[0;a_1,a_2,\ldots,a_t,\ldots]$ is connected with the limit behavior of the arithmetic mean $(a_1+a_2+\ldots+a_t)/t$. Particularly, N. Moshchevitin and A. Dushistova showed that if $a_1+a_2+\ldots+a_t<\kappa_1 t$, where $\kappa_1 = 2\log\bigl({\frac{1+\sqrt{5}}{2}}\bigr)/\log{2}= 1.3884\ldots$, then $?'(x)=+\infty$. They also proved that the constant $\kappa_1$ is non-improvable. We consider a dual problem: how small can be the quantity $a_1+a_2+\ldots+a_t-\kappa_1 t$ if $?'(x)=0$? We obtain the non-improvable estimates of this quantity.
\end{abstract}

\section{Introduction}
\subsection{The Minkowski function $?(x)$.}
For an arbitrary $x\in [0,1]$ we consider its continued fraction expansion
$$
x=[0;a_1,a_2,\ldots,a_n,\ldots]=  \cfrac{1}{a_1+\cfrac{1}{a_2+\cdots}}, \,\,\, a_j \in \mathbb{Z}_+
$$
with natural partial quotients $a_t$. This representation is infinite when $x\not\in\mathbb{Q}$ and finite for rational $x$. 
For irrational numbers the continued fraction representation is unique, however each rational $x$ has two different representations
$$
x=[0;a_1,a_2,\ldots,a_{n-1}, a_n]  
\,\,\,\,\,\text{and}\,\,\,\,\,
x=[0;a_1,a_2,\ldots,a_{n-1}, a_n-1,1]
,
\,\,\,\, \text{where}\,\,\,\, a_n \ge 2
.
$$
By 
$$\frac{p_k}{q_k}:=[0;a_1,\ldots,a_k]$$ 
we denote the $k$th convergent fraction to $x$.
By $B_n$ we denote the $n$th level of the Stern-Brocot tree, that is $$B_n:=\{ x=[0;a_1,\ldots,a_k]:a_1+\ldots+a_k=n+1\}.$$
In \cite{Mink} Minkowski introduced the function $?(x)$ which may be defined as the limit  distribution function
of sets $B_n$. This function was rediscovered several times and studied by many authors (see \cite{kinney},\cite{kesse},\cite{alkauskas},\cite{MD},\cite{paradis}).
For irrational $x=[0;a_1,a_2,\ldots,a_n,\ldots]$   the formula
\begin{equation} \label{eq:1}
?(x)=\sum\limits_{k=1}^{\infty} \frac{(-1)^{k+1}}{2^{a_1+\ldots+a_k-1}}
\end{equation}
introduced by Denjoy \cite{Den,Den1} and Salem \cite{Salem}
may be considered as one of the equivalent definitions of the function $?(x)$. If $x$ is rational, then the infinite series in (\ref{eq:1}) is replaced by a finite sum. Note that $?([0;a_1,\ldots,a_t+1])=?([0;a_1,\ldots,a_t,1])$ and hence $?(x)$ is well-defined for rational numbers too. It is known that $?(x)$ is a continuous strictly increasing function, also its derivative $?'(x)$, if exists, can take only two values -- $0$ and $+\infty$. Almost everywhere in [0; 1] in the sense of Lebesgue measure the derivative exists and equals $0$. If $x$ is rational, then $?'(x)$ also exists and equals zero.

\subsection{Notation and parameters}
We will denote the sequences by capital letters $A, B, C$ and their elements by the corresponding small letters $a_i, b_j,c_k$. All sequences of the present paper contain positive integers unless otherwise stated. For an arbitrary finite sequence $B=(b_1,b_2,\ldots,b_n)$ we denote
$$
\overleftarrow{B}=(b_n,b_{n-1},\ldots,b_1),\quad  S(B)=\sum\limits_{i=1}^{n}b_i.
$$
By $\langle A\rangle$ we denote the \textit{continuant} of (possibly empty) finite sequence $A=(a_1,\ldots,a_t)$. It is defined as follows: the continuant of the empty sequence equals $1$, $\langle a_1\rangle=a_1$, if $t\ge 2$ one has
\begin{equation}
\label{contrule}
\langle a_1,a_2,\ldots,a_t\rangle=a_t\langle a_1,a_2,\ldots,a_{t-1}\rangle+\langle a_1,a_2,\ldots,a_{t-2}\rangle.
\end{equation}
Note that the finite continued fraction $[0;a_1,\ldots,a_t]$ can be expressed using continuants:
\begin{equation}
[0;a_1,\ldots,a_t]=\frac{\langle a_2,\ldots,a_t\rangle}{\langle a_1,a_2,\ldots,a_t\rangle}.
\end{equation}
Rule (\ref{contrule}) can be generalized as follows:
\begin{equation}
\begin{split}
\label{contrulegen}
\langle a_1,a_2,\ldots,a_t,a_{t+1},\ldots,a_s\rangle=\langle a_1,a_2,\ldots,a_{t}\rangle\langle a_{t+1},\ldots,a_{s}\rangle+
\langle a_1,a_2,\ldots,a_{t-1}\rangle\langle a_{t+2},\ldots,a_{s}\rangle=\\
=\langle a_1,a_2,\ldots,a_{t}\rangle\langle a_{t+1},a_{t+2},\ldots,a_{s}\rangle(1+[0;a_{t-1},a_{t-2},\ldots,a_1][0;a_{t+2},a_{t+3},\ldots,a_s]).
\end{split}
\end{equation}
One can find more about the properties of continuants in \cite{Kan}.

For an irrational $x=[0;a_1,a_2,\ldots,a_n,\ldots]$ we consider the sum $S_x(t)$ of its partial quotients up to $t-$th:
$$
S_x(t)=a_1+a_2+\ldots+a_t.
$$
For an arbitrary finite sequence $B$ we will denote by $S(B)$ and $\Pi(B)$ the sum and the product of its elements respectively. 

Throughout the paper we always denote the sequence of partial quotients of $x$ by $a_1,a_2,\ldots,a_t,\ldots$ unless otherwise stated. We will also denote the sequence of $t$ first elements of this infinite sequence by $A_t$. Thus, $S_x(t)=S(A_t)$.

We also need the following constants
\begin{equation}
\label{Phikappa1}
\Phi = \frac{1+\sqrt{5}}{2}, \quad
\kappa_1 = \frac{2\log{\Phi}}{\log{2}}= 1.3884\ldots,
\end{equation}
\begin{equation}
\label{lamun}
\lambda_n = \frac{n+\sqrt{n^2+4}}{2}, 
\end{equation}
\begin{equation}
\label{kappa2}
\kappa_2 = \frac{4\log\lambda_5-5\log\lambda_4}{\log \lambda_5-\log\lambda_4-\log\sqrt{2}}=4.401\ldots ,
\end{equation}
\begin{equation}
\label{kappa4}
\kappa_4 =\sqrt{\frac{\kappa_1-1}{\log 2}}= 0.7486\ldots .
\end{equation} 
For an arbitrary sequence $A$ of length $t$ we denote $S(A)-\kappa_1t$ by $\varphi^{(1)}(A)$. For $x=[0;a_1,\ldots,a_t,\ldots]$ we denote $\varphi^{(1)}(A_t)$ by $\varphi^{(1)}_{x}(t)$.
\subsection{Critical values}
In the paper \cite{PaViBi} it was shown by J. Paradis, P. Viader, L. Bibiloni that the value of the derivative of the function $?(x)$ is connected with the limit behavior of $\frac{S_x(t)}{t}$. They showed that if $\frac{S_x(t)}{t}<\kappa_1$ and $?'(x)$ exists, then $?'(x)=+\infty$. On the other side, denote by $z\approx 5.319$ is the root of the equation $2\log(1+z)=z\log{2}$. If $\frac{S_x(t)}{t}\ge z$ and $?'(x)$ exists, then $?'(x)=0$. 

In the paper \cite{MD} A. Dushistova and N. Moshchevitin improved the results of \cite{PaViBi} and formulated the following two theorems.
\\
{\bf Theorem A.} \,\,\,
\\
\textit{(i) 
Let for real irrational $x\in(0, 1)$  the inequality $S_x(t)<\kappa_1t$ hold for all $t$ large enough. Then the derivative $?'(x)$ exists and equals $+\infty$.\\
(ii) For any positive $\varepsilon$ there exists an irrational number ${\bf x}\in (0,1)$, such that $?'(x)=0$ and the inequality $S_{x}(t)<(\kappa_1+\varepsilon)t$ holds for all $t$ large enough.
}
\\
{\bf Theorem B.} \,\,\,
\\
\textit{(i) 
Let for real irrational $x\in(0, 1)$  the inequality $S_x(t)>\kappa_2t$ hold for all $t$ large enough. Then the derivative $?'(x)$ exists and equals $0$.\\
(ii) For any positive $\varepsilon$ there exists an irrational number $x\in (0,1)$, such that $?'(x)=+\infty$ and the inequality $S_{x}(t)>(\kappa_2-\varepsilon)t$ holds for all $t$ large enough.
}

One can see that the constants $\kappa_1$ and $\kappa_2$ in theorems A and B are non-improvable.
\subsection{The dual problem}
In the paper \cite{MDK} the dual problem was considered. Suppose that $?'(x)=0$.  How small can be the difference $\varphi^{(1)}_{x}(t)=S_x(t)-\kappa_1t$? Statement (ii) of Theorem $A$ implies that $\varphi^{(1)}_{x}(t)$ can be less than $\varepsilon t$ for any positive $\varepsilon$. The first non-trivial estimate of $\varphi^{1}_{x}(t)$ was obtained in \cite{MDK} by A. Dushistova, I. Kan and N. Moshchevitin.\\
{\bf Theorem C.} \,\,\,
\\
\textit{(i) Let for irrational $x\in(0, 1)$  the derivative $?'(x)$ exists and $?'(x)=0$. Then for any $\varepsilon>0$ for all $t$ large enough one has
\begin{equation}
\label{k4kanest1}
\max_{u\le t}\varphi^{(1)}_{x}(u)=\max_{u\le t}\left(S_{x}(u)-\kappa_1 u\right) \ge (\kappa_4-\varepsilon)\sqrt{t\log t}.
\end{equation}
(ii) There exists an irrational ${\bf x}\in (0,1)$, such that $?'({\bf x})=0$ and or any $\varepsilon>0$ for all $t$ large enough one has
\begin{equation}
\label{k4kanest2}
\varphi^{(1)}_{x}(t)=S_{x}(t)-\kappa_1 t \le (2\sqrt{2+\varepsilon})\kappa_4 \sqrt{t\log t}.
\end{equation}
}
In the paper \cite{KanGay2} a strengthened version of the inequality (\ref{k4kanest1}) was obtained within the same condition
\begin{equation}
\label{k4kanestbetter}
\max_{u\le t}\varphi^{(1)}_{x}(u) \ge \biggl(\frac{2}{\sqrt{3}}+\varepsilon\biggr)\kappa_4\sqrt{t\log t}.
\end{equation}
Of course, one can ask the same question when $?'(x)=+\infty$.  How small can be the difference $\varphi^{(2)}_{x}(t):=\kappa_2t-S_x(t)$? The first result in this area was also obtained in \cite{MDK}. It was improved several times and for now the best known estimates are the following:\\
{\bf Theorem D.}\cite{KanGay1} \,\,\,
\\
\textit{(i) Let for real irrational $x\in(0, 1)$  the derivative $?'(x)$ exists and $?'(x)=+\infty$. Then for all $t$ large enough one has
\begin{equation}
\label{k5kanest1}
\max_{u\le t}\left(\kappa_2 u-S_{x}(u)\right) \ge 0.06222\sqrt{t}.
\end{equation}
(ii) There exists an irrational $x\in (0,1)$, such that $?'(x)=+\infty$ and for all $t$ large enough one has
\begin{equation}
\label{k5kanest2}
\kappa_2 t-S_{x}(t) \le 0.26489\sqrt{t}.
\end{equation}
}
\section{Main results}
Note that there are different quantities on the left-hand sides of the inequalities (\ref{k4kanest1}) and (\ref{k4kanest2}). One can say that the inequality  (\ref{k4kanest1}) considers the uniform behavior of $\varphi^{(1)}_{x}(t)$, whereas the inequality (\ref{k4kanest2}) deals with the local behavior of this quantity. In the present paper we consider upper and lower estimates in both cases. Our first theorem states that statement (ii) of Theorem C is non-improvable.
\begin{theor}
\label{maintheor1}
Let for irrational $x\in(0, 1)$  the derivative $?'(x)$ exists and $?'(x)=0$. Then for any $\varepsilon>0$ for infinitely many $t$ one has
\begin{equation}
\label{mainineqth1}
\varphi^{(1)}_{x}(t)\ge (2\sqrt{2}-\varepsilon)\kappa_4\sqrt{t\log t}.
\end{equation}
\end{theor}
Our second theorem provides optimal estimate of the uniform behavior of $\varphi^{(1)}_{x}(t)$.
\begin{theor}
\label{maintheor2}
(i)Let for irrational $x\in(0, 1)$  the derivative $?'(x)$ exists and $?'(x)=0$. Then for any $\varepsilon>0$ for all $t$ large enough one has
\begin{equation}
\label{mainineqth21}
\max_{u\le t}\varphi^{(1)}_{x}(u)\ge (\sqrt{2}-\varepsilon)\kappa_4\sqrt{t\log t}.
\end{equation}
(ii) For any $\varepsilon>0$ there exists an irrational $x\in (0,1)$ such that $?'(x)=0$ and for infinitely many $t$ one has
 \begin{equation}
\label{mainineqth22}
\max_{u\le t}\varphi^{(1)}_{x}(u)\le (\sqrt{2}+\varepsilon)\kappa_4\sqrt{t\log t}.
\end{equation}
\end{theor}
\section{Auxiliary lemmas}
\subsection{Increment lemmas and their corollaries}
The following two lemmas shed light on the connection between the behavior of the sum of partial quotients of $x$ and the value of the derivative $?'(x)$.
\begin{lem}(\cite{MDK}, Lemma 1)
\label{derlemlow}
For an irrational $x=[0;a_1,a_2,\ldots,a_t,\ldots]$ and for $\delta$ small in absolute value, there exists a natural $t=t(x,\delta)$ such that
\begin{equation}
\label{derlemloweq}
\frac{?(x+\delta)-?(x)}{\delta}\ge\frac{\langle A_t\rangle\langle A_{t-1}\rangle}{2^{S_x(t)+4}}.
\end{equation}
\end{lem}
\begin{lem}(\cite{MDK}, Lemma 2)
\label{derlemup}
For an irrational $x=[0;a_1,a_2,\ldots,a_t,\ldots]$ and for $\delta$ small in absolute value, there exists a natural $t=t(x,\delta)$ such that
\begin{equation}
\label{derlemupeq}
\frac{?(x+\delta)-?(x)}{\delta}\le\frac{\langle A_t\rangle^2}{2^{S_x(t)-2}}.
\end{equation}
\end{lem}
It is not convenient for us that the numerators of the right-hand sides of (\ref{derlemloweq}) and (\ref{derlemupeq}) do not coincide. That is why we prove a ''symmetric'' corollary of Lemmas \ref{derlemlow} and \ref{derlemup}
\begin{lem}
\label{derlemsym}
For an irrational $x=[0;a_1,a_2,\ldots,a_t,\ldots]$ the derivative $?'(x)$ equals zero if and only if
\begin{equation}
\label{derlemsymeq}
\lim\limits_{t\to\infty} \frac{\langle A_t\rangle}{\sqrt{2}^{S_x(t)}}=0.
\end{equation}
\end{lem}
\begin{proof}
($\Leftarrow$) Follows immediately from Lemma \ref{derlemup}.\\
($\Rightarrow$) It follows from Lemma \ref{derlemlow} that 
\begin{equation}
\label{derlemsymeqmin1}
\lim\limits_{t\to\infty} \frac{\langle A_{t-1}\rangle}{\sqrt{2}^{S_x(t)}}=0,
\end{equation}
Suppose that $\frac{\langle A_t\rangle}{\sqrt{2}^{S_x(t)}}$ does not tend to zero as $t\to\infty$. That is, there exists a positive constant $c$ such that 
\begin{equation}
\label{derlemsymcontr}
\frac{\langle A_t\rangle}{\sqrt{2}^{S_x(t)}}>c
\end{equation}
for infinitely many $t$. Consider $N\in\mathbb{N}$ such that for any $n\ge N$ 
$$
\frac{\langle A_{n-1}\rangle}{\sqrt{2}^{S_x(n)}}<\frac{c}{100}.
$$
Consider an arbitrary integer $t>N+a_N$ such that  the inequality (\ref{derlemsymcontr}) holds. Then
$$
\frac{(a_{t}+1)c}{100}>\frac{(a_t+1)\langle A_{t-1}\rangle}{\sqrt{2}^{S_x(t)}}>\frac{\langle A_t\rangle}{\sqrt{2}^{S_x(t)}}>c
$$
Thus, $a_t\ge 100$. On the other hand,
$$
\frac{\sqrt{2}^{a_t}c}{(a_{t-1}+1)(a_t+1)}<\frac{\sqrt{2}^{a_t}}{a_{t-1}+1}\frac{\langle A_{t-1}\rangle}{\sqrt{2}^{S_x(t)}}<\frac{\langle A_{t-2}\rangle}{\sqrt{2}^{S_x(t-1)}}<\frac{c}{100}
$$
Therefore, we have $100\sqrt{2}^{a_t}<(a_{t-1}+1)(a_t+1)$. As $a_t\ge 100$, one can easily see that $a_{t-1}>a_t+1$. From (\ref{derlemsymcontr}) one can also derive that 
$$
\frac{\langle A_{t-1}\rangle}{\sqrt{2}^{S_x(t-1)}}>\frac{a_t+1}{\sqrt{2}^{a_t}}\frac{\langle A_{t-1}\rangle}{\sqrt{2}^{S_x(t-1)}}>\frac{\langle A_t\rangle}{\sqrt{2}^{S_x(t)}}>c.
$$
Repeating the same argument $t-N$ times we obtain that $a_{N}>a_{t}+t-N> a_t+a_N$ and therefore we come to a contradiction.
\end{proof}
\subsection{Continuant lower estimate}
The following lemma is a useful tool to estimate the values of continuants, most of whose elements are equal to $1$. We introduce some notation first. Having a sequence $A_t$, denote by $w(A_t)$ the number of its elements greater than $1$. The set of such elements forms the sequence that we denote by $d_{A_t}(1), d_{A_t}(2),\ldots,d_{A_t}(w(A_t))$.
\begin{lem}
\label{prodlem}
For an arbitrary continuant $\langle A_t\rangle=\langle a_1,a_2,\ldots,a_t\rangle$ one has
\begin{equation}
\label{contlowest}
\langle a_1,a_2,\ldots,a_t\rangle\ge \frac{1}{2} \Phi^t \prod\limits_{i=1}^{w(A_t)} \biggl(\frac{d_{A_t}(i)}{4}\biggr).
\end{equation}
\end{lem}
\begin{proof}
We prove by induction on $w(A_t)$. If $w(A_t)$ equals zero, then $\langle A_t\rangle=F_{t+1}$ -- $(t+1)$-th Fibonacci number and one can easily verify (\ref{contlowest}) using Binet's formula.

Now suppose that the inequality  (\ref{contlowest}) holds for $w(A_t)=n$. Consider an arbitrary continuant $\langle A_t\rangle$ having $w(A_t)=n+1$. Let us also consider the continuant $\langle A'_t\rangle$, which is obtained from $\langle A_t\rangle$ by replacing one of its elements greater than one by $1$. Of course, $w(A'_t)=n$ and one has 
\begin{equation}
\label{contlowestn}
\langle A_t\rangle=\frac{\langle A_t\rangle}{\langle A'_t\rangle}\langle A'_t\rangle\ge\frac{\langle A_t\rangle}{\langle A'_t\rangle} \frac{1}{2} \Phi^t \prod\limits_{i=1}^{n} \biggl(\frac{d_{A'_t}(i)}{4}\biggr).
\end{equation}
Applying (\ref{contrulegen}), one can easily see that for arbitrary finite sequences $A$ and $B$ one has
$$
\langle A, x, B\rangle=\langle A, x\rangle\langle B\rangle(1+[0;x,\overleftarrow{A}][0;B])=x\langle A\rangle\langle B\rangle(1+[0;\overleftarrow{A}][0;x])(1+[0;x,\overleftarrow{A}][0;B]).
$$
Hence,
\begin{equation}
\label{axba1b}
\frac{\langle A, x, B\rangle}{\langle A, 1, B\rangle}=x\frac{1+[0;x,\overleftarrow{A}][0;B]}{1+[0;1,\overleftarrow{A}][0;B]}\frac{1+[0;\overleftarrow{A}][0;x]}{1+[0;\overleftarrow{A}]}
\ge
 x\frac{1}{2}\cdot\frac{1}{2}=\frac{x}{4}.
\end{equation}
Applying the lower estimate of $\frac{\langle A_t\rangle}{\langle A'_t\rangle}$ from (\ref{axba1b}) to (\ref{contlowestn}), we prove the induction step. The lemma is proved.
\end{proof}
\begin{lem}
\label{greaterzero}
Suppose that $?'(x)=0$. Then $\varphi^{(1)}_x(t)>0$ for all $t$ large enough.
\end{lem}
\begin{proof}
From Lemma \ref{derlemsym} one has 
$$
\frac{\langle A_t\rangle}{\sqrt{2}^{S_x(t)}}>\frac{1}{2}.
$$
Using an obvious estimate $\langle A_t\rangle>\frac{\Phi^t}{2}$ from (\ref{contlowest}), one has
$$
\frac{1}{2}<\frac{\langle A_t\rangle}{\sqrt{2}^{S_x(t)}}<\frac{1}{2}\frac{\Phi^t}{\sqrt{2}^{\kappa_1 t+\varphi^{(1)}_x(t)}}.
$$
As $\sqrt{2}^{\kappa_1}=\Phi$, we obtain the statement of the lemma. 
\end{proof}
\begin{lem}
\label{sumprodlem}
Consider two arbitrary real numbers $\beta>\alpha\ge 3$ and a real number $s\ge \beta$. Let $R(s,\alpha,\beta)$ be the set of all finite sequences $R=(r_1,\ldots,r_k)$ of real numbers such that $\alpha\le r_i\le \beta$ for all $i\le k$ and $S(R)=s$. Then
\begin{equation}
\label{sumprodlemeq}
\min\limits_{R\in R(s,\alpha,\beta)}\Pi(R)\ge \beta^{\bigl[\frac{s}{\beta}\bigr]}
\end{equation}
\end{lem}
\begin{proof}
Note that the number $k$ is not fixed in the definition of $ R(s,\alpha,\beta)$. Denote 
$$
f(s,\alpha,\beta)=\min\limits_{R\in R(s,\alpha,\beta)}\Pi(R). 
$$
It is clear that the function $f(s,\alpha,\beta)$ is monotonic in the first argument. Hence, without loss of generality one can say that $\frac{s}{\beta}\in\mathbb{Z}$. Consider the sequence $R_{\beta}=(\underbrace{\beta,\beta,\ldots,\beta}_{s/\beta \text{ times}})$. Using the compactness argument, one can easily see that there exists a sequence $R_0=(r^0_1,\ldots,r^0_m)\in R(s,\alpha,\beta)$  such that $\Pi(R_0)=f(s,\alpha,\beta)$. Suppose that  $R_0\ne R_{\beta}$. One can easily show that $R_0$ cannot contain more than two elements not equal to $\alpha$ or $\beta$. Indeed, if $\alpha<r^0_i\le r^0_j<\beta$, then there exists $\delta>0$ such that $r^0_i-\delta>\alpha$ and $r^0_j+\delta<\beta$. As $(r^0_i-\delta)(r^0_j+\delta)<r^0_i r^0_j$, one can see that
$$
\Pi(r^0_1,\ldots,r^0_{i-1},r^0_i-\delta,r^0_{i+1},\ldots,r^0_{j-1},r^0_j+\delta,r^0_{j+1},\ldots,r^0_m)<
\Pi(r^0_1,\ldots,r^0_{i-1},r^0_i,r^0_{i+1},\ldots,r^0_{j-1},r^0_j,r^0_{j+1},\ldots,r^0_m)
$$ 
and we obtain a contradiction with the definition of $R_0$.

On the other hand, it follows from the definition of $R_0$ that
\begin{equation}
\label{pir0pirm}
\Pi(R_0)\le \Pi(R_{\beta}).
\end{equation}
Without loss of generality one can assume that $R_0$ does not contain elements equal to $\beta$. Indeed, if we remove all such elements from $R_0$ and the same number of elements from $R_{\beta}$, the inequality (\ref{pir0pirm}) will be still satisfied. Thus, $R_0$ has the form $R_0=(x,\underbrace{\alpha,\alpha,\ldots,\alpha}_{(s-x)/\alpha \text{ times}})$ up to transposition of elements. Here $\alpha\le x<\beta$.

Denote $n=\frac{s-x}{\alpha}$. Inequality (\ref{pir0pirm}) can be written as
\begin{equation}
\label{pir0pirm2}
\alpha^n x\le\beta^{\frac{\alpha n+x}{\beta}}.
\end{equation}
But in fact, the opposite inequality is true for all $\alpha\le x\le\beta$. Taking into account the fact that $\alpha^{\beta}>\beta^{\alpha}$ for $\beta>\alpha\ge 3$, one can easily verify that 
\begin{equation}
\label{pir0pirm2}
\alpha^n x>\beta^{\frac{\alpha n+x}{\beta}}
\end{equation}
for $x=\alpha$ and $x=\beta$. As  $\beta^{\frac{\alpha n+x}{\beta}}$ is a convex downward function of $x$, one can deduce that it lies below the linear function $\alpha^n x$ for all $\alpha\le x\le\beta$. Thus, we obtain a contradiction with (\ref{pir0pirm}) and the lemma is proved.
\end{proof}
\subsection{Elimination of small elements greater than $1$}
\begin{lem}
\label{unitvar}
Let $A,C$ be arbitrary (possibly empty) sequences of positive integers. Let $B$ be a symmetric sequence. Consider two arbitrary integers $p\ge m\ge 1$. Then 
\begin{equation}
\langle A,1,B,p+m-1,C\rangle\le\langle A,m,B,p,C\rangle.
\end{equation}
\begin{proof}
Denote $q=\frac{p+m}{2}$. Consider the function 
$$
f(x)=\langle A,q+x,B,q-x,C\rangle,
$$ 
where $x$ run through the set of integers if $p+m$ is even and runs through the set of half-integers otherwise. One can see that $f(x)$ is a quadratic polynomial with negative leading coefficient. The maximum of $f(x)$ is attained at the point\footnote{see \cite{Gay1}, Lemma 5.4 for computational details}
$$
x_m=\frac{[0;\overleftarrow{B}]-[0;B]+[0;C]-[0;\overleftarrow{A}]}{2}.
$$
As $B$ is a symmetric sequence, $[0;\overleftarrow{B}]=[0;B]$ and therefore $|x_m|\le\frac{1}{2}$. Hence, as $p\ge m\ge 1$, one can easily see that 
$$
\langle A,1,B,p+m-1,C\rangle=f\biggl(-\frac{p+m-2}{2}\biggr)\le f\biggl(-\frac{p-m}{2}\biggr)=\langle A,m,B,p,C\rangle.
$$
\end{proof}
\end{lem}
\begin{lem}
\label{no2345}
Suppose  that $?'(x)=0$. Then there exists an irrational number $y=[0;b_1,b_2,\ldots,b_t,\ldots]$ such that:
\begin{enumerate}
\item{$?'(y)=0$.}
\item{For all $i\in\mathbb{N}$ either $b_i=1$ or $b_i\ge 12$.}
\item{For all $i\in\mathbb{N}$ one has $\varphi^{(1)}_y(i)\le\varphi^{(1)}_x(i)$.}
\end{enumerate}
\end{lem}
\begin{proof}
First, let us eliminate all elements equal to $2$ from the sequence $a_1,a_2,\ldots$. Denote by $s_1$ the smallest index such that $a_{s_1}=2$. Denote by $t_1$ the smallest index greater than $s_1$ such that $a_{t_1}>1$. Now the procedure is repeated recursively:
$$
s_i=\min\{n: n>t_{i-1}, a_{n}=2\},\quad t_i=\min\{n: n>s_{i}, a_{n}>1\}
$$
Thus, we obtain the two (possibly infinite) growing sequences $s_1<t_1<s_2<t_2<\ldots$. Note that if $s_j<n<t_j$ for some $j$, then $a_n=1$. Define the irrational number $x'=[0;a'_1,a'_2,\ldots,a'_t,\ldots]$ as follows:
$$
a'_n=
\begin{cases}
a_n-1=1\quad\text{if $n=s_i$ for some $i\in\mathbb{N}$}\\
a_n+1\quad\text{if $n=t_i$ for some $i\in\mathbb{N}$}\\
a_n\quad \text{otherwise}
\end{cases}
$$
One can easily see from the definition of $x'$ that $\varphi^{(1)}_{x'}(t)\le\varphi^{(1)}_x(t)$ for all $t\ge 1$. Let us now show that $?'(x')=0$.

It follows from Lemma \ref{unitvar} that $\langle a'_1,a'_2,\ldots,a'_t\rangle:=\langle A'_t\rangle\le\langle A_t\rangle$ for all $t\ge 1$. On the other hand $|{S_{x'}(t)}-{S_{x}(t)}|\le 1$. Hence, as  
$$
\lim\limits_{t\to\infty} \frac{\langle A_t\rangle}{\sqrt{2}^{S_{x}(t)}}=0
$$
by Lemma \ref{derlemsym}, we obtain that
$$
\lim\limits_{t\to\infty} \frac{\langle A'_t\rangle}{\sqrt{2}^{S_{x'}(t)}}=0
$$
and therefore $?'(x')=0$. Using the same argument, one can eliminate the elements equal to $3,4,\ldots,12$ from the sequence $a'_1,a'_2,\ldots$. The lemma is proved.
\end{proof}
\begin{lem}
\label{aquxfoll}
Suppose that $?'(x)=0$ and all partial quotients of $x$ are either equal to $1$ or greater than $11$. There exists $T\in\mathbb{N}$ such that for all $t>T$ the inequality $\varphi^{(1)}_x(t)>3 w(A_t)$ holds. 
\end{lem}
\begin{proof}
Lemma \ref{derlemsym} implies that there exists $T$ such that $\forall t>T$ one has
$$
\frac{\langle A_t\rangle}{\sqrt{2}^{S_x(t)}}<\frac{1}{2}.
$$
Hence, as $\frac{x}{4}\ge 3$ for $x\ge 12$ and $\sqrt{2}^{\kappa_1}=\Phi$, one has by Lemma \ref{prodlem}
\begin{equation}
\label{auxcorrineq}
\frac{1}{2}>\frac{\langle A_t\rangle}{\sqrt{2}^{S_x(t)}}\ge \frac{1}{2}\frac{\Phi^t 3^{w(A_t)}}{\sqrt{2}^{\kappa_1 t+\varphi^{(1)}_x(t)}}=
\frac{1}{2}\frac{3^{w(A_t)}}{\sqrt{2}^{\varphi^{(1)}_x(t)}}.
\end{equation}
Statement of the lemma immediately follows from (\ref{auxcorrineq}).
\end{proof}
\section{Blocks structure}
\subsection{Parameters introduction}
Let $x=[0;a_1,a_2,\ldots,a_t,\ldots]$ be an irrational number such that $?'(x)=0$. By Lemma \ref{no2345}, without loss of generality one can say that either $a_i=1$ or $a_i\ge 6$ for all $i\in\mathbb{N}$. Throughout the remaining part of the paper we consider $\varepsilon$ as a fixed positive real number from the statements of Theorems \ref{maintheor1} and \ref{maintheor2}. Let $\lambda=\lambda(\varepsilon)$ be an arbitrary rational number such that $1>\lambda>1-\varepsilon^6.$ Define the following integer constants:
\begin{equation}
\label{constdef}
M=\frac{10\log{\varepsilon}}{\log{\lambda}},\ \ P=\biggl[\frac{\log{6}}{\log{(1+\varepsilon^2)}}\biggr]+1,\ \ N=2M(P+2).
\end{equation}
Now we select an integer parameter $t_0$ large enough such that $(1-\lambda)\lambda^N t_0>\frac{t_0}{\log{t_0}}$ and 
\begin{equation}
\label{t0progr}
\lambda^N t_0\in\mathbb{Z}.
\end{equation}
Denote $t_i=\lambda^i t_0$, where $1\le i\le N$. We define $B_i=(a_{t_i+1},a_{t_i+2},\ldots,a_{t_{i-1}})$ -- the $i$-th block, where $1\le i\le N$. Denote $t_{N+1}=0$ and $B_{N+1}=(a_{t_{N+1}+1},\ldots,a_{t_N})$. Thus, we have
$$
[0;A_{t_0}]=[0;a_1,\ldots,a_{t_0}]=[0;B_{N+1},B_N,\ldots,B_1].
$$
One can easily see that $S_x(t_0)=\sum\limits_{i=1}^{N+1}S(B_i)$ and  $\varphi^{(1)}_x(t_0)=\sum\limits_{i=1}^{N+1}\varphi^{(1)}_x(B_i)$.

For each block $B_i$ denote its greatest element by $M_i$ and the index of such element by $m_i$ (if the greatest element is not unique, we take the rightmost one). Thus, $a_{m_i}=M_i$. Denote $c_k=\frac{M_k}{\sqrt{t_{k-1}\log{t_0}}}$. Let us also consider for each $1\le i\le N+1$ the short block $B'_i=(a_{t_i+1},\ldots,a_{m_i-1})$. Note that
\begin{equation}
\label{phimisum}
\varphi^{(1)}_x(m_i)=\varphi^{(1)}(B_{N+1})+\varphi^{(1)}(B_N)+\ldots+\varphi^{(1)}(B_{i+1})+\varphi^{(1)}(B'_i)+(M_i-\kappa_1).
\end{equation}
For each $1\le i\le N+1$ define the real numbers $f_i$ and $f'_i$ from the following identities:
\begin{equation}
\label{fidef}
\langle B_i\rangle=\sqrt{2}^{S(B_i)+f_i\sqrt{t_{i-1}\log{t_0}}},\quad \langle B'_i\rangle=\sqrt{2}^{S(B'_i)+f'_i\sqrt{t_{i-1}\log{t_0}}}.
\end{equation}
\subsection{Lower estimate of $\varphi^{(1)}(B_k)$}
\begin{lem}
\label{fklogtkneg}
Suppose that $?'(x)=0$. Then for all $1\le i\le N$ one has
\begin{equation}
\label{sumfkneg}
f'_i\sqrt{t_{i-1}\log{t_0}}+\sum\limits_{k=i+1}^{N+1} f_k\sqrt{t_{k-1}\log{t_0}} <0.
\end{equation}
\end{lem}
\begin{proof}
As $?'(x)=0$, by Lemma \ref{derlemsym}, without loss of generality one can say that for all $1\le i\le N$ the inequality
$$
\frac{\langle B_{N+1},B_N,\ldots,B_{i+1},B'_i\rangle}{\sqrt{2}^{S(B_{N+1})+S(B_N)+\ldots+S(B_{i+1})+S(B'_i)}}<1
$$
is satisfied. Using (\ref{contrulegen}) we obtain
\begin{equation}
\label{fklogtknegineq}
\frac{\langle B_{N+1}\rangle\langle B_N\rangle\ldots\langle B_{i+1}\rangle\langle B'_i\rangle}{\sqrt{2}^{S(B_{N+1})+S(B_N)+\ldots+S(B_{i+1})+S(B'_i)}}<1.
\end{equation}
Substituting (\ref{fidef}) to (\ref{fklogtknegineq}) and taking logarithm of both parts, we get the statement of the lemma.
\end{proof}
\begin{lem}
\label{uniformlem}
Suppose that $?'(x)=0$. If the inequality  
\begin{equation}
\label{nonuniformall}
|\varphi^{(1)}(B_k)|\ge\kappa_1(t_{k-1}-t_k)\varepsilon^5,
\end{equation}
holds for some $1\le k\le N+1$, then one has:
\begin{equation}
\max_{t_N\le u\le t_0}\varphi^{(1)}_{x}(u)\ge t_0^{0.9}.
\end{equation}
\end{lem}
\begin{proof}
Suppose that 
$$
\varphi^{(1)}(B_k)\ge\kappa_1(t_{k-1}-t_k)\varepsilon^5.
$$
As,
$$
\varphi^{(1)}_x(t_{k-1})=\varphi^{(1)}_x(t_{k})+\varphi^{(1)}(B_k),
$$
using the fact that $\varphi^{(1)}_x(t_{k})>0$ by Lemma \ref{greaterzero}, we have
$$
\max_{t_N\le u\le t_0}\varphi^{(1)}_{x}(u)\ge\varphi^{(1)}_x(t_{k-1})>\varphi^{(1)}(B_k)\ge\kappa_1(t_{k-1}-t_k)\varepsilon^5=\kappa_1\lambda^{k-1}(1-\lambda)t_0>t_0^{0.9}.
$$
On the other hand, if
$$
\varphi^{(1)}(B_k)\le-\kappa_1(t_{k-1}-t_k)\varepsilon^5,
$$
we again use the fact that $\varphi^{(1)}_x(t_{k-1})>0$ and obtain
$$
\max_{t_N\le u\le t_0}\varphi^{(1)}_{x}(u)\ge\varphi^{(1)}_x(t_{k})=\varphi^{(1)}_x(t_{k-1})-\varphi^{(1)}(B_k)\ge\kappa_1(t_{k-1}-t_k)\varepsilon^5>t_0^{0.9}.
$$
\end{proof}
One can easily deduce from Lemma \ref{uniformlem} that if the inequality (\ref{nonuniformall}) holds, then the inequalities (\ref{mainineqth1}) and (\ref{mainineqth21}) are also satisfied. Therefore, throughout the remaining part of the paper we will assume that 
\begin{equation}
\label{sbkass}
S(B_k)=\kappa_1(t_{k-1}-t_k)(1+o(\varepsilon^4))
\end{equation}
for all $1\le k\le N+1$.
\begin{lem}
\label{philowest}
Suppose that $?'(x)=0$. Then for all $1\le k\le N+1$ one has
\begin{equation}
\label{genest}
\varphi^{(1)}(B_k)\ge\frac{(\kappa_1-1)(t_{k-1}-t_k)\sqrt{\log t_0}(1+o(\varepsilon^4))}{c_k\sqrt{t_{k-1}}\log{2}}-f_k\sqrt{t_{k-1}\log{t_0}}.
\end{equation}
\end{lem}
\begin{proof}
We recall that $w(B_k)$ is the number of elements of the block $B_k$ which are greater than $1$ and $d_{B_k}(1),\ldots,d_{B_k}(w(B_k))$ is the sequence of such elements. It follows from Lemma \ref{uniformlem} and Lemma \ref{aquxfoll} that
$$
\sum\limits_{i=1}^{w(B_k)}d_{B_k}(i)=(\kappa_1-1)(t_{k-1}-t_k)(1+o(\varepsilon^4)).
$$
Therefore, as $d_{B_k}(i)\ge 12$ for all $i$, one has $\frac{d_{B_k}(i)}{4}\ge 3$. Now we can obtain a lower estimate of 
$$
\prod\limits_{i=1}^{w(B_k)} \biggl(\frac{d_{B_k}(i)}{4}\biggr)
$$
applying Lemma \ref{sumprodlem} for $s=\frac{1}{4}(\kappa_1-1)(t_{k-1}-t_k)(1+o(\varepsilon^4))$, $\alpha=3$, and $\beta=c_k\sqrt{t_{k-1}
\log{t_0}}$. We have
\begin{equation}
\label{prodlowest}
\prod\limits_{i=1}^{w(B_k)} \biggl(\frac{d_{B_k}(i)}{4}\biggr)\ge \biggl(\frac{c_k\sqrt{t_{k-1}\log{t_0}}}{4}\biggr)^{\frac{(\kappa_1-1)(t_{k-1}-t_k)(1+o(\varepsilon^4))}{c_k\sqrt{t_{k-1}\log{t_0}}}}.
\end{equation}
Substituting the estimate (\ref{prodlowest}) to (\ref{contlowest}) we obtain
\begin{equation}
\label{biest}
\langle B_k\rangle \ge \Phi^{t_{k-1}-t_k} \biggl(\frac{c_k\sqrt{t_{k-1}\log{t_0}}}{4}\biggr)^{\frac{(\kappa_1-1)(t_{k-1}-t_k)(1+o(\varepsilon^4))}{c_k\sqrt{t_{k-1}\log{t_0}}}}.
\end{equation}
Taking into account that $\sqrt{2}^{\kappa_1}=\Phi$, from (\ref{fidef}), (\ref{biest}) and Lemma \ref{derlemsym} we get:
\begin{equation}
\label{beforelog}
\biggl(\frac{9c_i\sqrt{t_{k-1}\log{t_0}}}{16}\biggr)^{\frac{(\kappa_1-1)(t_{k-1}-t_k)(1+o(\varepsilon^4))}{c_i\sqrt{t_{k-1}\log{t_0}}}}\le
\sqrt{2}^{\varphi^{(1)}(B_k)+f_i\sqrt{t_{k-1}\log{t_0}}}.
\end{equation}
Taking logarithms of both parts of (\ref{beforelog}), after some transformations we obtain the inequality
\begin{equation}
\frac{(\kappa_1-1)(t_{k-1}-t_k)(1+o(\varepsilon^4))}{c_i\sqrt{t_{k-1}\log{t_0}}}\frac{\log t_0}{2}\le(\varphi^{(1)}(B_k)+f_k\sqrt{t_{k-1}\log{t_0}})\log\sqrt{2}
\end{equation}
which is equivalent to (\ref{genest}). Lemma is proved.
\end{proof}
\subsection{Main estimate}
For $k\le N$ one can write  (\ref{genest}) as
\begin{equation}
\label{lenest}
\varphi^{(1)}(B_k)\ge\biggl(\frac{(\kappa_1-1)(1-\lambda)(1+o(\varepsilon^4))}{c_k\log{2}}-f_k\biggr)\sqrt{t_{k-1}\log{t_0}}.
\end{equation}
For $k=N+1$ we have
\begin{equation}
\label{eqnplus1est}
\varphi^{(1)}(B_{N+1})\ge\biggl(\frac{(\kappa_1-1)(1+o(\varepsilon^4))}{c_{N+1}\log{2}}-f_{N+1}\biggr)\sqrt{t_{N}\log{t_0}}.
\end{equation}
Let the greatest element of the short block $B'_k$ be equal to $c'_k\sqrt{t_{k-1}\log{t}}$. Using the argument from Lemma \ref{philowest}, one can deduce the following lower estimate of $\varphi^{(1)}(B'_k)$:
\begin{equation}
\label{genest2}
\varphi^{(1)}(B'_k)\ge\frac{S(B'_k)\sqrt{\log t_0}(1+o(\varepsilon^4))}{c'_k\sqrt{t_{k-1}}\log{2}}-f'_k\sqrt{t_{k-1}\log{t_0}}.
\end{equation}

\begin{lem}
\label{mainlowest}
Suppose that $?'(x)=0$. Then for all $1\le i\le N$ one has
\begin{equation}
\label{mainlowineq}
\varphi^{(1)}_x(m_i)\ge\Biggl(\frac{(\kappa_1-1)(1-\lambda)}{\log{2}}\biggl(\sum\limits_{k=i+1}^{N} \frac{(\sqrt{\lambda})^{k-i}}{c_k}\biggr)+c_i\Biggr)\sqrt{t_{i-1}\log{t_0}}\bigl(1+o(\varepsilon^4)\bigr).
\end{equation}
\end{lem}
\begin{proof}
Substituting the estimates (\ref{genest}) for $k=N,N-1,\ldots,i+1$, (\ref{eqnplus1est}), and (\ref{genest2}) to (\ref{phimisum}) and taking into account the inequality (\ref{sumfkneg}) we obtain:
\begin{equation}
\begin{split}
\label{bigsumi}
\varphi^{(1)}_x(m_i)\ge(1+o(\varepsilon^4))\biggl(\frac{\kappa_1-1}{\log{2}c_{N+1}}\sqrt{t_{N}\log{t_0}}+
\frac{S(B_i)\sqrt{\log t_0}}{\log{2}c'_i\sqrt{t_{i-1}}}+\\
\frac{(\kappa_1-1)\sqrt{\log t_0}}{\log{2}}\sum\limits_{k=i+1}^{N} \frac{t_{k-1}-t_k}{c_k\sqrt{t_{k-1}}}\biggr)+(c_i\sqrt{t_{i-1}\log{t_0}}-\kappa_1).
\end{split}
\end{equation}
Using the trivial estimates
$$
\frac{S(B_i)\sqrt{\log t_0}}{\log{2}c'_i\sqrt{t_{i-1}}}>0, \quad \frac{\kappa_1-1}{\log{2}c_{N+1}}\sqrt{t_{N}\log{t_0}}>0
$$
we obtain
\begin{equation}
\begin{split}
\label{bigsumi2}
\varphi^{(1)}_x(m_i)\ge(1+o(\varepsilon^4))\biggl(
\frac{(\kappa_1-1)\sqrt{\log t_0}}{\log{2}}\sum\limits_{k=i+1}^{N} \frac{t_{k-1}-t_k}{c_k\sqrt{t_{k-1}}}\biggr)+(c_i\sqrt{t_{i-1}\log{t_0}}-\kappa_1).
\end{split}
\end{equation}
Taking into account the fact that $t_{k}=\lambda^kt_0$, we immediately obtain (\ref{mainlowineq}). The lemma is proved.
\end{proof}
Inequality (\ref{mainlowineq}) is the key tool that we will use in proofs of Theorem \ref{maintheor1} and the first statement of Theorem \ref{maintheor2}. Let us simplify it using a new notation. Denote
\begin{equation}
\label{alphadef}
\alpha=\frac{(\kappa_1-1)}{\log{2}},\quad \eta=\frac{1}{\alpha}.
\end{equation}
Then, one can rewrite (\ref{mainlowineq}) as follows:
\begin{equation}
\begin{split}
\label{bigsumfullsimple}
\varphi^{(1)}_x(m_i)\ge\biggl((1-\lambda)\sum\limits_{k=i+1}^{N} \frac{(\sqrt{\lambda})^{k-i}}{c_k}+\eta c_i\biggr)\alpha\sqrt{t_{i-1}\log{t_0}}.\bigl(1+o(\varepsilon^4)\bigr).
\end{split}
\end{equation}
\section{Key lemmas}
The inequality (\ref{bigsumfullsimple}) reduces the estimation of $\max\limits_{1\le i\le N}\varphi^{(1)}_x(m_i)$ to the problem of finding maximum of the following quantity:
$$
(1-\lambda)\sum\limits_{k=i+1}^{N} \frac{(\sqrt{\lambda})^{k-i}}{c_k}+\eta c_i, \quad i=1,2,\ldots,N.
$$
Note that this problem does not deal with the Minkowski function, continued fractions etc. It is purely combinatorial. The following lemma allows us to estimate the desired maximum.
\begin{lem}
\label{sumlem}
Let $\eta, c_1,c_2,\ldots, c_N$ be arbitrary positive real numbers. Define the real numbers $\varphi_i$ as follows:
\begin{equation}
\label{phikdef}
\varphi_i=(1-\lambda)\sum\limits_{k=i+1}^{N}\frac{\sqrt{\lambda}^{k-i}}{c_k}+\eta c_i.
\end{equation}
Then the following inequality holds:
\begin{equation}
\label{keyineq2}
\max\limits_{1\le i\le N}\varphi_i\ge\sqrt{8\eta}(1+o(\varepsilon)).
\end{equation}
\end{lem}
\begin{proof}
The proof of the lemma will be splitted into several steps. We also recall that the constants $M$, $N$, and $P$ used in our argument are defined in (\ref{constdef}).
\begin{lem}
\label{phi1lem}
Suppose that the inequality (\ref{keyineq2}) is not satisfied. Then there exists a natural number $i_1\le M$ such that $c_{i_1}\ge\frac{1}{2\sqrt{\eta}}$. Moreover, for all $i\le N$ one has
\begin{equation}
\label{3c0}
c_i<\frac{3}{\sqrt{\eta}}.
\end{equation}
\end{lem}
\begin{proof}
Suppose the contrary. Let us estimate $\varphi_1$ from below:
\begin{equation}
\label{phi1est}
\begin{split}
\varphi_1\ge(1-\lambda)\sum\limits_{k=1}^{N}\frac{\sqrt{\lambda}^{k}}{c_{k+1}}+\eta c_1\ge(1-\lambda)\sum\limits_{k=1}^{M}\frac{\sqrt{\lambda}^{k}}{c_{k+1}}\ge 2\sqrt{\eta}(1-\lambda)\sum\limits_{k=1}^{M}\sqrt{\lambda}^{k}=\\
= 2\sqrt{\eta}(1-\lambda)\sqrt{\lambda}\frac{1-\sqrt{\lambda}^M}{1-\sqrt{\lambda}}=2\sqrt{\eta}(1+\sqrt{\lambda})(1+o(\varepsilon^4))>\sqrt{8\eta}(1+o(\varepsilon^4)).
\end{split}
\end{equation}
We obtain a contradiction with (\ref{keyineq2}). The estimate (\ref{3c0}) comes from the trivial inequality $\varphi_i>\eta c_i$. Lemma is proved.
\end{proof}
\begin{lem}
\label{iterlem2}
Suppose that the inequality (\ref{keyineq2}) is not satisfied. Then for all $i_m<N-M,\  m\ge 1$ there exists a number $i_m<i_{m+1}<i_m+M$ such that $c_{i_{m+1}}>(1+\varepsilon^2)c_{i_m}.$
\end{lem}
\begin{proof}
Suppose the contrary. Let the inequality $c_{i_m+j}<c_{i_m}(1+\varepsilon^2)$ be satisfied for all $1\le j\le M$. Then, using the argument from Lemma \ref{phi1lem} we obtain
\begin{equation}
\label{phikest}
\begin{split}
\varphi_{i_m}\ge(1-\lambda)\sum\limits_{j=1}^{M}\frac{\sqrt{\lambda}^{j}}{c_{i_m+j}}+\eta c_{i_m}\ge\frac{1-\lambda}{1+\varepsilon^2}\sum\limits_{j=1}^{M}\frac{\sqrt{\lambda}^{j}}{c_{i_m}}+\eta c_{i_m}\ge\\
\biggl(\frac{2}{(1+\varepsilon^2)c_{i_m}}+\eta c_{i_m}\biggr)(1+o(\varepsilon^4))\ge \sqrt{\frac{8\eta}{1+\varepsilon^2}}\bigl(1+o(\varepsilon^4)\bigr).
\end{split}
\end{equation}
In the last ''$\ge$'' of (\ref{phikest}) we use the Cauchy-Schwarz inequality. We come to a contradiction with  (\ref{keyineq2}). Lemma is proved.
\end{proof}
Now we are ready to prove Lemma \ref{sumlem}. Suppose that the inequality (\ref{keyineq2}) is not satisfied. By Lemma \ref{phi1lem} there exists $1\le i_1<M$ such that $c_i\ge\frac{1}{2\sqrt{\eta}}$. Applying Lemma \ref{iterlem2}\\ $P=\biggl[\frac{\log{6}}{\log{(1+\varepsilon^2)}}\biggr]+1$ times, we obtain the number $c_{i_{P+1}}$ such that  $c_{i_{P+1}}>\frac{1}{2\sqrt{\eta}}(1+\varepsilon^2)^P \eta>\frac{3}{\sqrt{\eta}}$. We come to a contradiction with (\ref{3c0}). Lemma is proved.
\end{proof}
One can rewrite (\ref{bigsumfullsimple}) as 
\begin{equation}
\begin{split}
\label{bigsumfullsimpleprime}
\varphi^{(1)}_x(m_i)\ge\sqrt{\lambda}^i\biggl((1-\lambda)\sum\limits_{k=i+1}^{N} \frac{(\sqrt{\lambda})^{k-i}}{c_k}+\eta c_i\biggr)\alpha\sqrt{t_0\log{t_0}}(1+o(\varepsilon^4)).
\end{split}
\end{equation}
Note that $\sqrt{\lambda}=1+o(\varepsilon^5)$. As Theorem \ref{maintheor2} requires us to show that $\max\limits_{1\le i\le N}\varphi^{(1)}_x(m_i)$ is greater than $\sqrt{t_0\log{t_0}}$ multiplied by some constant factor, we need to estimate the maximum of the following quantity
$$
\lambda^i\biggl((1-\lambda)\sum\limits_{k=i+1}^{N} \frac{(\sqrt{\lambda})^{k-i}}{c_k}+\eta c_i\biggr), \quad i=1,2,\ldots,N.
$$
This estimate is provided by the following lemma.
\begin{lem}
\label{sumlemprime}
Let $C=(c_1,c_2,\ldots, c_N)$ be an arbitrary sequence of non-negative real numbers. Let $\eta$ be an arbitrary positive real number. Define the numbers $\varphi_i$ using (\ref{phikdef}). Define the numbers $\varphi'_i$ as follows:
\begin{equation}
\label{phikprimedef}
\varphi'_i(C)=\sqrt{\lambda}^i\varphi_i=(1-\lambda)\sum\limits_{k=i+1}^{N}\frac{\sqrt{\lambda}^{k}}{c_k}+\sqrt{\lambda}^i\eta c_i.
\end{equation}
Then one has:
\begin{equation}
\label{keyineq2prime}
\max\limits_{1\le i\le N}\varphi'_i(C)\ge\sqrt{2\eta}(1+o(\varepsilon)).
\end{equation}
\end{lem}
\begin{proof}
The proof of the lemma will be also splitted into several steps. First, using the substitution $d_j=\sqrt{\lambda}^j c_j$, we write (\ref{phikprimedef}) as
\begin{equation}
\label{recsource}
\tilde{\varphi}'_i(D):=\varphi'_i(C)=(1-\lambda)\sum\limits_{k=i+1}^{N}\frac{\lambda^{k}}{d_k}+\eta d_i.
\end{equation}
Here $D=(d_1,d_2,\ldots,d_N)$. Denote $\tilde{\varphi}'_{max}(D)=\max\limits_{1\le k\le N}\tilde{\varphi}'_k(D)$. In order to prove the lemma it is enough for us to show that 
\begin{equation}
\label{PhiCmin}
\min\limits_{D\in\mathbb{R}^N_{+}}\tilde{\varphi}_{max}(D)=\sqrt{2\eta}(1+o(\varepsilon)).
\end{equation}
Denote the minimum of (\ref{PhiCmin}) by $y_{min}$.
\begin{lem}
\label{minmaxphilem}
Suppose that $\tilde{\varphi}'_{max}(D)=y_{min}$ for some $D\in\mathbb{R}^N_{+}$. Then for all $1\le k\le N$ one has $\tilde{\varphi}'_k(D)=y_{min}$.
\end{lem}
\begin{proof}
Suppose that $\tilde{\varphi}'_{max}(D)=y_{min}$, but for some $i$ one has $\tilde{\varphi}'_i(D)<y_{min}$. We call the index $n$ \textit{minimizing} if $\varphi'_n(D)=y_{min}$. Let $k$ be the largest non-minimizing index. Without loss of generality one can say that all indices less than $k$ are not minimizing too. Indeed, consider the sequence
$$
D'=(d_1,\ldots,d_{k-1},d_k+\delta,d_{k+1},\ldots,d_N),
$$
where $\delta>0$ is some small parameter. One can easily see that $\tilde{\varphi}'_i(D')<\tilde{\varphi}'_i(D)$ for $i<k$, $\tilde{\varphi}'_i(D')>\tilde{\varphi}'_i(D)$ for i$=k$ and 
$\tilde{\varphi}'_i(D')=\tilde{\varphi}'_i(D)$ for $i>k$. As $k$ is a non-minimizing index, there exists $\delta>0$ such that $\tilde{\varphi}'_i(D')<y_{min}$ for all $i\le k$.

Thus, $k+1$ is the smallest minimizing index. As all indices less than $k+1$ are not minimizing, there exists $\delta>0$ such that for the sequence
$$
D''=(d_1,\ldots,d_k,d_{k+1}-\delta,d_{k+2},\ldots,d_N)
$$
one has $y_{min}>\tilde{\varphi}'_i(D'')>\tilde{\varphi}'_i(D)$ for $i\le k$ and $\tilde{\varphi}'_i(D'')<\tilde{\varphi}'_i(D)=y_{min}$ for $i=k+1$. Thus, we obtained the sequence whose smallest minimizing index is at least $k+2$. Repeating this argument we obtain the sequence $D^{(N)}=(d^{(N)}_1,d^{(N)}_2,\ldots,d^{(N)}_N)$ with the smallest minimizing index equal to $N$. One can easily see that for the sequence
$$
D'^{(N)}=(d^{(N)}_1,d^{(N)}_2,\ldots,d^{(N)}_N-\delta)
$$
for $\delta>0$ small enough one has $\tilde{\varphi}'_{max}(D'^{(N)})<\tilde{\varphi}'_{max}(D^{(N)})=y_{min}$ and we obtain a contradiction with the definition of $y_{min}$. Lemma is proved.
\end{proof}
\begin{lem}
There exists a unique sequence $D=(d_1,d_2,\ldots, d_N)$ such that $\tilde{\varphi}'_{max}(D)=y_{min}$. The elements of this sequence satisfy the recurrent equation 
\begin{equation}
\label{recal}
d_{k+1}=\frac{d_k+\sqrt{d_k^2+\frac{4(1-\lambda)\lambda^{k+1}}{\eta}}}{2}
\end{equation}
with the initial condition $d_1=0$.
\end{lem}
\begin{proof}
Lemma \ref{minmaxphilem} implies that $\tilde{\varphi}'_N(D)=y_{min}$. This fact yields a linear equation $\eta d_N=y_{min}$ from which $d_N$ is uniquely defined. Then we substitute $d_N$ to the equation $\tilde{\varphi}'_{N-1}(D)=y_{min}$ and find $d_{N-1}$ etc. Therefore the sequence $D$ is uniquely defined.

Suppose that $d_1>0$. But if we decrease $d_1$, we would also decrease $\tilde{\varphi}'_1(D)$, but $\tilde{\varphi}'_{max}(D)$ would be still equal to $y_{min}$. We obtain a contradiction with the uniqueness of $D$. Finally, from the equation $\tilde{\varphi}'_{k+1}(D)=\tilde{\varphi}'_{k}(D)$ one can derive
\begin{equation}
\label{ckqadreq}
\eta(d_{k+1}-d_k)=(1-\lambda)\frac{\lambda^{k+1}}{d_{k+1}}.
\end{equation}
Considering (\ref{ckqadreq}) as the quadratic equation on $d_{k+1}$ and choosing the positive root, we obtain (\ref{recal}). Lemma is proved.
\end{proof}
Thus, $y_{min}=\frac{d_N}{\eta}$ where $d_N$ can be evaluated from the recurrent equations (\ref{recal}).
Multiplying both sides of (\ref{recal}) by $\sqrt{\eta}$, we obtain
\begin{equation}
\label{recalnew}
\sqrt{\eta}d_{k+1}=\frac{\sqrt{\eta}d_k+\sqrt{(\sqrt{\eta}d_k)^2+4(1-\lambda)\lambda^{k+1}}}{2}.
\end{equation}
Put $e_k=\sqrt{\eta}d_k$. Using the introduced notation, one can write (\ref{recalnew}) as follows:
\begin{equation}
\label{rec1}
e_{k+1}=\frac{e_k+\sqrt{e_k^2+4(1-\lambda)\lambda^{k+1}}}{2}.
\end{equation}
Denote\footnote{Equation (\ref{rec1}) is equivalent to $\frac{e_{k+1}-e_k}{\delta_{k+1}}=\frac{1}{e_{k+1}}$. Thus, (\ref{rec1}) might be considered as numerical integration of the differential equation $y'=\frac{1}{y}$ on the non-uniform grid $X_i$. I am thankful to I. Mitrofanov who drew my attention to this fact.} $\delta_k=(1-\lambda)\lambda^{k}$ и $X_n=\sum\limits_{k=1}^n\delta_n$.

\begin{lem}
\label{dnconv}
For $n\ge 1$ one has $e_n<\sqrt{2X_n}$ and $e_{n+1}-e_n>\sqrt{2X_{n+2}}-\sqrt{2X_{n+1}}$.
\end{lem}
\begin{proof}
The first statement is proved by induction. For $n=1$ one can easily verify the inequality. As $\sqrt{x}$ is a convex function, one can easily see that
\begin{equation}
e_{n+1}=\frac{e_n+\sqrt{e_n^2+4\delta_{n+1}}}{2}<\frac{\sqrt{2X_n}+\sqrt{2X_n+4\delta_{n+1}}}{2}<\sqrt{2X_n+2\delta_{n+1}}=\sqrt{2X_{n+1}}.
\end{equation}

Now we prove the second statement. Note that
\begin{equation}
e_{n+1}-e_n=\frac{\sqrt{e_n^2+4\delta_{n+1}}-e_n}{2}=\frac{2\delta_{n+1}}{e_n+\sqrt{e_n^2+4\delta_{n+1}}}=\frac{\delta_{n+1}}{e_{n+1}}.
\end{equation}
On the other hand, as we already showed, $\sqrt{2X_{n+2}}>\sqrt{2X_{n+1}}>e_{n+1}$ and therefore
\begin{equation}
\sqrt{2X_{n+2}}-\sqrt{2X_{n+1}}=\frac{2\delta_{n+2}}{\sqrt{2X_{n+2}}+\sqrt{2X_{n+1}}}<\frac{\delta_{n+1}}{e_{n+1}}=e_{n+1}-e_n.
\end{equation}
Lemma is proved.
\end{proof}
Now we are ready to prove Lemma \ref{sumlemprime}. By Lemma \ref{dnconv}, $\sqrt{2X_{n+1}}-e_n$ forms a decreasing sequence of positive real numbers. Hence 
$$
0<\sqrt{2X_N}-e_{N-1}<\sqrt{2X_2}-e_1<2\sqrt{X_2}=\sqrt{2(1-\lambda)(\lambda+\lambda^2)}<2\sqrt{1-\lambda}=o(\varepsilon^{2}).
$$
On the other hand, $2X_N=2\sum\limits_{k=1}^N (1-\lambda)\lambda^{k}=2(1-\lambda^{N+1})=2+o(\varepsilon^5)$.
Thus, $e_{N-1}=\sqrt{2}+o(\varepsilon^{2})$. It also follows from (\ref{rec1}) that $e_{N}=\sqrt{2}+o(\varepsilon^{2})$. Hence,
$$
y_{min}=\eta{d_N}=\sqrt{\eta}e_N=\sqrt{2\eta}+o(\varepsilon^{2}).
$$
Lemma is proved.
\end{proof}
\section{Proof of Theorem \ref{maintheor1} and the first statement of Theorem \ref{maintheor2}}
Now we are ready to prove Theorem \ref{maintheor1}.
\begin{proof}
From (\ref{bigsumfullsimple}) and Lemma \ref{sumlem} one can deduce that for all $t_0$ large enough there exist integer numbers $i$ and $m_i$ satisfying $1\le i\le  N$ and $\frac{t_0}{\log{t_0}}<m_i\le t_0$ such that
\begin{equation}
\label{phi1final}
\varphi^{(1)}_x(m_i)\ge\sqrt{8\eta}\alpha\sqrt{t_{i-1}\log{t_0}}(1+o(\varepsilon))=2\sqrt{2}\kappa_4\sqrt{t_{i-1}\log{t_0}}(1+o(\varepsilon)).
\end{equation}
As $t_i<m_i\le t_{i-1}$ and $\frac{t_{i}}{t_{i-1}}=\lambda=1+o(\varepsilon^5)$, from (\ref{phi1final}) one has
$$
\varphi^{(1)}_x(m_i)\ge2\sqrt{2}\kappa_4\sqrt{m_i\log{m_i}}(1+o(\varepsilon)).
$$
Theorem is proved.
\end{proof}
Let us now prove the first statement of Theorem \ref{maintheor2}.
\begin{proof}
It follows from  (\ref{bigsumfullsimpleprime}) and Lemma \ref{sumlemprime} that for all $t_0$ large enough there exist integer numbers $i$ and $m_i$ satisfying $1\le i\le  N$ and $\frac{t_0}{\log{t_0}}<m_i\le t_0$ such that
$$
\sqrt{\lambda}^i\varphi^{(1)}_x(m_i)\ge\sqrt{2\eta}\alpha\sqrt{t_i\log{t_0}}(1+o(\varepsilon))=\sqrt{2}\kappa_4\sqrt{t_i\log{t_0}}(1+o(\varepsilon)).
$$
In other words,
$$
\max_{u\le t_0}\varphi^{(1)}_{x}(u)\ge\varphi^{(1)}_x(m_i)\ge\sqrt{2}\kappa_4\sqrt{t_0\log{t_0}}(1+o(\varepsilon)).
$$
The first statement of Theorem \ref{maintheor2} is proved.
\end{proof}

\section{Proof of the second statement of Theorem \ref{maintheor2}}
\subsection{Superblocks definition}
\label{superbsec}
\begin{proof}
In this chapter we will construct an irrational number $x=[0;a_1,\ldots,a_n,\ldots]$ that will satisfy the conditions of the second statement of Theorem \ref{maintheor2}. The continued fraction of $x$ will have the form
\begin{equation}
\label{xstruc0}
x=[0;\mathfrak{B}^{(0)},\mathfrak{B}^{(1)},\mathfrak{B}^{(2)},\ldots,\mathfrak{B}^{(n)},\ldots],
\end{equation}
where the segments $\mathfrak{B}^{(i)}$ will be defined later. We will call these segments \textit{superblocks}. Recall that the sequence $d_k$ is defined from the equations (\ref{recal}) with the initial condition $d_1=0$. It will be convenient for us to modify the first element of this sequence. We set $d_1=d_2=\sqrt{\frac{(1-\lambda)\lambda}{\eta}}=o(\varepsilon^5)$.

Let us now define the sequence $T_i$ that plays the key role in our construction. We choose an arbitrary integer $T_1$, satisfying $\lambda^NT_1>\frac{T_1}{\log{T_1}}$. Then, if the number $T_{i-1}$ is defined, we put $T_{i}=\bigl[\frac{T_{i-1}}{\lambda^N}\bigr]$. For each $i\ge 1$ put $t^{(i)}_0=T_i$.

Now we describe the construction of the superblock $\mathfrak{B}^{(i)}$ for an arbitrary positive integer $i$. For each $k$ such that $1\le k\le N$ we select three natural numbers $m^{(i)}_k, n^{(i)}_k$ и $t^{(i)}_k$ from the following conditions:
\begin{equation}
\begin{split}
\label{mknkdef}
d_k\sqrt{T_i\log{T_i}}\le m^{(i)}_k \le d_k\sqrt{T_i\log{T_i}}(1+\varepsilon^4)\\
\frac{d_k}{\kappa_1-1}\sqrt{T_i\log{T_i}}\le n^{(i)}_k\le \frac{d_k}{\kappa_1-1}\sqrt{T_i\log{T_i}}(1+\varepsilon^4)\\
\frac{\log{T_i}}{\log{2}}(1+\frac{\varepsilon}{8}-\varepsilon^3)\le m^{(i)}_k+n^{(i)}_k-\kappa_1 (n^{(i)}_k+1) \le\frac{\log{T_i}}{\log{2}}(1+\frac{\varepsilon}{8}+\varepsilon^3)\\
t^{(i)}_k=\lambda^k t^{(i)}_0+\theta (n^{(i)}_k+1),\ \text{where}\ |\theta|\le\frac{1}{2}, \quad (n^{(i)}_k+1)\mid (t^{(i)}_{k-1}-t^{(i)}_k).
\end{split}
\end{equation}
One can easily see that the numbers $m^{(i)}_k, n^{(i)}_k, t^{(i)}_k$, satisfying conditions (\ref{mknkdef}) always exists. Consider the block $B^{(i)}_k=(a_{t^{(i)}_k+1},\ldots,a_{t^{(i)}_{k-1}})$ having the following structure:
\begin{equation}
\label{B1kstruct}
B^{(i)}_k=(m^{(i)}_k,\underbrace{1,1,\ldots,1}_{n^{(i)}_k \text{\ numbers}},m^{(i)}_k,\underbrace{1,1,\ldots,1}_{n^{(i)}_k \text{\ numbers}},\ldots,m^{(i)}_k,\underbrace{1,1,\ldots,1}_{n^{(i)}_k \text{\ numbers}}).
\end{equation} 
We denote the sequence of $N$ blocks $(B^{(i)}_N,B^{(i)}_{N-1},\ldots,B^{(i)}_1)$  by $\mathfrak{B}^{(i)}$. For the initial superblock $\mathfrak{B}^{(0)}=(a_1,a_2,\ldots,a_{t^{(1)}_N})$ we set all its elements to be equal to $1$. Thus, the construction of the continued fraction (\ref{xstruc0}) is fully described.
Note that the initial superblock $\mathfrak{B}^{(0)}$ has fixed length and therefore does not affect neither the value of $?'(x)$ nor the behavior of $\varphi^{(1)}_x(t)$ as $t$ grows. 
\subsection{$?'(x)=0$}
As we already mentioned it is enough to show that for $x'=[0;\mathfrak{B}^{(1)},\mathfrak{B}^{(2)},\ldots,\mathfrak{B}^{(n)},\ldots])$ one has $?'(x')=0$. Denote the elements of continued fraction expansion of $x'$ by $a'_1,a'_2,\ldots$. By Lemma \ref{derlemsym} it is enough to show that the function
\begin{equation}
\label{bikfract2}
f_{x'}(t)=\frac{\langle a'_1,a'_2,\ldots,a'_t\rangle}{\sqrt{2}^{a'_1+a'_2+\ldots+a'_t}}
\end{equation}
tends to $0$ as $t\to\infty$.  One can easily see that the function $f_{x'}(t)$ grows if $a'_t=1$ and decreases if $a'_t\ge 5$. Of course, all $m^{(i)}_k$ are greater than $5$. Thus it is enough to consider only the case when $a'_t=1$, but $a'_{t+1}\ge 5$. In this case the continuant $\langle a'_1,a'_2,\ldots,a'_t\rangle$ consists of the sequences of the form  $(m^{(i)}_l,\underbrace{1,1,\ldots,1}_{n^{(i)}_l \text{\ numbers}})$. From the formula $\langle A,B\rangle\le 2\langle A\rangle\langle B\rangle$ and Lemma \ref{derlemsym} one can easily see that if we show that
\begin{equation}
\label{bikfract3orig}
\frac{2\langle m^{(i)}_l,\underbrace{1,1,\ldots,1}_{n^{(i)}_l \text{\ numbers}}\rangle}{\sqrt{2}^{m^{(i)}_l+n^{(i)}_l}}<\frac{1}{2},
\end{equation}
we will obtain the fact that $?'(x)=0$. Note that $4\langle m^{(i)}_l,\underbrace{1,1,\ldots,1}_{n^{(i)}_l \text{\ numbers}}\rangle<8m^{(i)}_l\Phi^{n^{(i)}_l}$. From the definition (\ref{mknkdef}) one can easily see that $m^{(i)}_l+n^{(i)}_l=\kappa_1 n^{(i)}_l+\frac{\log{T_i}}{\log{2}}(1+\frac{\varepsilon}{8}+o(\varepsilon^2))$. Thus, (\ref{bikfract3orig}) is equivalent to the following:
\begin{equation}
\label{bikfract3}
\frac{8m^{(i)}_l\Phi^{n^{(i)}_l}}{\sqrt{2}^{\kappa_1 n^{(i)}_l+\frac{\log{T_i}}{\log{2}}(1+\frac{\varepsilon}{8}+o(\varepsilon^2))}}=\frac{8m^{(i)}_l}{\sqrt{2}^{\frac{\log{T_i}}{\log{2}}(1+\frac{\varepsilon}{8}+o(\varepsilon^2))}}<1.
\end{equation}
Taking logarithm of both sides of (\ref{bikfract3}) and substituting $m^{(i)}_l$ from (\ref{mknkdef}) , we obtain that
\begin{equation}
\label{bikfract4}
\frac{\log{T_i}}{2}\Bigl(1+o(\varepsilon^2)\Bigr)<\frac{\log{T_i}}{\log{2}}\Bigl(1+\frac{\varepsilon}{8}+o(\varepsilon^2)\Bigr)\log{\sqrt{2}}=\frac{\log{T_i}}{2}\Bigl(1+\frac{\varepsilon}{8}+o(\varepsilon^2)\Bigr).
\end{equation}
Therefore, $?'(x)=0$.
\subsubsection{The inequality (\ref{mainineqth22}) is satisfied}
We will show that for $x$ that we built in Section \ref{superbsec} for all $t=T_i$, $i\ge 1$ the inequality (\ref{mainineqth22}) is satisfied. This inequality is equivalent to the following:
\begin{equation}
\label{mainineqprimecontranew}
\varphi^{(1)}_x(\nu)=S_x(\nu)-\kappa_1 \nu\le (\sqrt{2}+\varepsilon)\kappa_4\sqrt{t\log t} \quad \forall \nu\le t.
\end{equation}
We will prove (\ref{mainineqprimecontranew}) by induction on $i$. Suppose that $t=T_1=t^{(1)}_0$. As the segment $\mathfrak{B}^{(0)}=(a_1,a_2,\ldots,a_{ t^{(1)}_N)}$ has length $\bigl(o(\varepsilon^6)\bigr)t$, we will not take it into account in our further argument. 

One can easily see that $\varphi_1(\nu)>\varphi_1(\nu-1)$ if and only if $a_\nu>1$. Thus, it is enough to verify the inequality (\ref{mainineqprimecontranew})  only in the case when $a_\nu>1$. Suppose that $t^{(1)}_{i}<\nu\le t^{(1)}_{i-1}$ It follows from the definition (\ref{mknkdef}) that
$$
m^{(1)}_i+n^{(1)}_i-\kappa_1 (n^{(1)}_i+1)>0.
$$
Hence for any finite sequence $B$ and for any $1\le j\le N$ one has
$$
\varphi^{(1)}(B,\underbrace{1,\ldots,1}_{n^{(1)}_j \text{\ numbers}},m^{(1)}_j)>\varphi^{(1)}(B).
$$ 
Thus, it is enough to verify (\ref{mainineqprimecontranew}) only for the largest  $t^{(1)}_{k}<\nu\le t^{(1)}_{k-1}$ such that $a_{\nu}>1$. In this case we have
\begin{equation}
\begin{split}
\label{mainineqprimecontranew2}
\varphi^{(1)}_x(\nu)\le\sum\limits_{i=k}^N\varphi^{(1)}(B^{(1)}_i)+m^{(1)}_k=
\sum\limits_{i=k}^N\biggl(\frac{T_1(\lambda^{i-1}-\lambda^i)}{n^{(1)}_i+1}\frac{\log{T_1}}{\log{2}}\Bigl(1+\frac{\varepsilon}{8}\Bigr)\biggr)\bigl(1+o(\varepsilon^2)\bigr)+m^{(1)}_k .
\end{split}
\end{equation}
This can be written as follows:
\begin{equation}
\label{mainineqprimecontranew3}
\varphi^{(1)}_x(\nu)\le\Biggl(\frac{(1-\lambda)(\kappa_1-1)}{\log{2}}\Bigl(1+\frac{\varepsilon}{8}\Bigr)\sum\limits_{i=k+1}^N\frac{\lambda^{i-k}}{d_i}+d_k\Biggr)\sqrt{T_1\log{T_1}}\bigl(1+o(\varepsilon^2)\bigr).
\end{equation}
Denote $\alpha=\frac{\kappa_1-1}{\log{2}}(1+\frac{\varepsilon}{8})$, $\eta=\frac{1}{\alpha}$. Using the introduced notation we have
\begin{equation}
\label{mainineqprimecontranew31}
\varphi^{(1)}_x(\nu)\le\Biggl((1-\lambda)\sum\limits_{i=k+1}^N\frac{\lambda^{i-k}}{d_i}+\eta d_k\Biggr)\alpha\sqrt{T_1\log{T_1}}\bigl(1+o(\varepsilon^2)\bigr).
\end{equation}
We recall that $d_1, d_2,\ldots,d_N$ is the minimizing sequence for (\ref{recsource}). Thus from Lemma \ref{minmaxphilem} we obtain that
\begin{equation}
\label{mainineqprimecontranew32}
\varphi^{(1)}_x(\nu)\le\alpha\sqrt{2\eta}\sqrt{T_1\log{T_1}}(1+o(\varepsilon))=\sqrt{2\alpha}\sqrt{T_1\log{T_1}}(1+o(\varepsilon))\le  \biggl(\sqrt{2}+\frac{\varepsilon}{2} \biggr)\kappa_4\sqrt{T_1\log{T_1}}.
\end{equation}
Thus the inequality (\ref{mainineqprimecontranew}) is satisfied for $t=T_1$. Using the same argument one can derive that for all $n\in\mathbb{N}$ one has
\begin{equation}
\label{mainineqprimecontranew3fori}
\varphi^{(1)}(\mathfrak{B^{(n)}})\le \biggl(\sqrt{2}+\frac{\varepsilon}{2}\biggr)\kappa_4\sqrt{T_n\log{T_n}}
\end{equation}
and for all $1\le k\le N$, $n\in\mathbb{N}$
\begin{equation}
\label{mainineqprimecontranew3fori2}
\sum\limits_{i=k}^N\varphi^{(1)}(B^{(n)}_i)+m^{(n)}_k\le \biggl(\sqrt{2}+\frac{\varepsilon}{2}\biggr)\kappa_4\sqrt{T_n\log{T_n}}.
\end{equation}
Now we show that (\ref{mainineqprimecontranew}) is satisfied for $t=T_n$ when $n\ge 2$. Using the same argument, one can deduce
\begin{equation}
\label{mainineqprimecontranew2long}
\varphi^{(1)}_x(\nu)\le\sum\limits_{i=1}^{n-1}\varphi_1(\mathfrak{B}^{(i)})+\sum\limits_{i=k}^N\varphi^{(1)}(B^{(n)}_i)+m^{(n)}_k
\end{equation}
for some $1\le k\le N$. From (\ref{mainineqprimecontranew3fori}) it follows that
\begin{equation}
\begin{split}
\label{mainineqprimecontranew3fori3}
\sum\limits_{i=1}^{n-1}\varphi_1(\mathfrak{B}^{(i)})\le\biggl(\sqrt{2}+\frac{\varepsilon}{2}\biggr)\kappa_4\sum\limits_{i=1}^{n-1}\sqrt{T_i\log{T_i}}\le
\biggl(\sqrt{2}+\frac{\varepsilon}{2}\biggr)\kappa_4\sqrt{\log{T_n}}\sum\limits_{i=1}^{n-1}\sqrt{T_i}=\\
=\biggl(\sqrt{2}+\frac{\varepsilon}{2}\biggr)\kappa_4\sqrt{T_n\log{T_n}}\sum\limits_{i=1}^{n-1}\sqrt{\lambda}^{Ni}\le
\biggl(\sqrt{2}+\frac{\varepsilon}{2}\biggr)\kappa_4\sqrt{T_n\log{T_n}}\frac{\sqrt{\lambda}^{N}}{1-\sqrt{\lambda}^{N}}=\biggl(\sqrt{T_n\log{T_n}}\biggr)o(\varepsilon^4).
\end{split}
\end{equation}
Substituting the estimates (\ref{mainineqprimecontranew3fori2}) and (\ref{mainineqprimecontranew3fori3}) to (\ref{mainineqprimecontranew2long}) we obtain that
\begin{equation}
\label{finalest}
\varphi^{(1)}_x(\nu)\le \biggl(\sqrt{2}+\frac{2\varepsilon}{3}\biggr)\kappa_4\sqrt{T_n\log{T_n}}
\end{equation}
and the second statement of Theorem \ref{maintheor2} is proved.
\end{proof}
\textbf{Acknowledgements:} I would like to thank Igor Kan for fruitful discussions. Most results of this paper are based on his ideas.

Dmitry Gayfulin,\\
Steklov Mathematical Institute of Russian Academy of Sciences\\
ul. Gubkina, 8, Moscow, Russia, 119991\\
\textit{gamak.57.msk@gmail.com}\\

\end{document}